\documentclass[11pt]{article}

\usepackage[hypertex]{hyperref}      

\usepackage[active]{srcltx}          

\usepackage[active]{srcltx}
\usepackage{amsmath,amssymb}

\usepackage{amssymb}
\usepackage{amscd}
\usepackage{amsmath}
\usepackage{amsfonts}
\usepackage{amsthm}
\usepackage{mathrsfs}

\newtheorem{theo}{Theorem}[section]

\newtheorem{prop}[theo]{Proposition}

\theoremstyle{remark}

\newcommand{\R}{{\mathbb{R}}}

\newcommand{\C}{{\mathbb{C}}}
\newcommand{\Z}{{\mathbb{Z}}}

\newcommand{\la}{\lambda} 
\newcommand{\al}{\alpha} 
\newcommand{\be}{\beta} 
\newcommand{\Om}{\Omega} 
 
\newcommand{\ga}{\gamma}
\newcommand{\om}{\omega} 
\newcommand{\Si}{\Sigma} 
\newcommand{\si}{\sigma} 
\newcommand{\Ga}{\Gamma}

\newcommand{\Ci}{{\mathcal{C}}^{\infty}}

\newcommand{\Ad}{\operatorname{Ad}} 
\newcommand{\aad}{\operatorname{ad}} 
\newcommand{\hol}{\operatorname{Hol}} 
\newcommand{\id}{\operatorname{id}} 

\newcommand{\Hom}{\operatorname{Hom}}
\newcommand{\End}{\operatorname{End}}

\newcommand{\Hilb}{{\mathcal{H}}}   

\newcommand{\SU}{\operatorname{SU}}  

\renewcommand{\top}{\operatorname{top}}

\newcommand{\Ac}{{\mathcal{A}}} 
\newcommand{\Teich}{{\mathcal{T}}} 

\newcommand{\Diff}{\operatorname{Diff}}   

\newcommand{\mo}{{\mathcal{M}}}

\newcommand{\op}{{\operatorname{Op}}}  

\newcommand{\difh}{D^{\operatorname{hol}}}
\newcommand{\I}{{\mathcal{I}}}

\newcommand{\toep}{\operatorname{Toep}}

\begin{document}

\title{Asymptotic properties of the quantum representations of the
  mapping class group}

\author{Laurent CHARLES \footnote{Institut de
    Math{\'e}matiques de Jussieu (UMR 7586), Universit{\'e} Pierre et
    Marie Curie -- Paris 6, Paris, F-75005 France.}}

\bibliographystyle{plain}

\maketitle

\begin{abstract}
We establish various results on the large level limit of projective
quantum representations of surface mapping class groups obtained by
quantizing moduli spaces of flat $SU(n)$-bundle. Working with the
metaplectic correction, we proved that these projective
representations lift to asymptotic representations. We show that
the operators in these representations are Fourier integral
operators and determine explicitly their canonical relations and symbols. We deduce from these facts the Egorov property and the asymptotic unitarity, two results already proved by J.E. Andersen. Furthermore we show under a transversality
assumption that the characters of these representations have an
asymptotic expansion. The leading order term of this expansion agrees
with the formula derived heuristically by E. Witten in {\em Quantum field theory and the Jones polynomial}.
\end{abstract}

Quantum Chern-Simons theory was introduced twenty years ago by
Witten \cite{Wi} and Reshetikhin-Turaev \cite{ReTu}. It provides finite dimensional projective representations of surface mapping class
groups. These representations may be defined with combinatorial-topological
methods or through conformal field theory. Another approach, that we
will follow, consists in quantizing the moduli spaces of flat
principal bundles with a given structure group.

More precisely, consider a closed surface $\Si$, a compact Lie group $G$ and
an integer $k$, called the level. The moduli space of flat $G$-principal bundles
over $\Si$ is the base of a Hermitian line bundle called
the Chern-Simons bundle. The quantum Hilbert space associated to
($\Si$, $G$, $k$) is the space of holomorphic sections of the $k$-th
tensor power of the Chern-Simons bundle. Here
the complex structure is induced by a class of complex structures in
the Teichm{\"u}ller space of $\Si$. So we have a vector bundle over the Teichm{\"u}ller space, called the
Verlinde bundle, whose fibers are the various Hilbert spaces. 
Hitchin \cite{Hi} and Axelrod-Della Pietra-Witten \cite{AxWi}
proved that the Verlinde bundle has a natural projectively flat connection. Using this connection, we can identify the different fibers through parallel
transport. The connection being equivariant with respect to the
mapping class group action, we obtain a projective representation of
this group.  

In this paper, we study the asymptotic properties of the connection and
deduce several facts on the mapping class group representations. Here
and in the sequel, asymptotic always refer to the large level limit. 

We consider only the case where $G$ is the special
unitary group $SU(n)$ with $n \geqslant 2$, the genus $g$ of $\Si$ is
larger than 2 and the moduli space is the space of
$G$-principal flat bundles whose
holonomy around a given marked point is a fixed generator of the center of
$G$. With this last condition the moduli space is smooth.  
The relevant mapping class group
here is $\Ga_g^1$, the group of isotopy classes of orientation
preserving homeomorphisms that are the identity on a fixed neighborhood
of the marked point. Using the connection of the Verlinde bundle defined in \cite{Hi}, we
obtain a projective representation of $\Ga_g^1$. 

We work with the metaplectic correction. This does not change the
Verlinde bundle, neither the projective representation. Nevertheless
the connection is modified by a multiple of identity and its natural
automorphisms group is an extension of the mapping class group
$$ 1 \rightarrow \Z/2 \Z \rightarrow \tilde{\Ga}_g^1 \rightarrow
\Ga_g^1 \rightarrow 1.$$ 
This modified connection has been introduced in \cite{AnGaLa}, the
comparison with the one of \cite{Hi} is made in
the present paper.

Our first semi-classical result says that the modified connection is asymptotically flat in
the sense that its curvature is bounded on compact subsets by a
constant times the inverse of the level, cf. theorem \ref{theo:curvature}. As a consequence we obtain an asymptotic representation of
$\tilde{\Ga}_g^1$. More precisely we can lift the projective representation
in such a way that the projective multiplier of any
two elements of $\tilde{\Ga}_g^1$ is bounded by the inverse of the
level. 

Second we prove that the operators in these asymptotic representations
are Fourier integral operators, the underlying canonical relations
being given by the action of the mapping class group on the
Chern-Simons bundle and the principal symbols by the action on the
half-form bundle, cf. theorems \ref{theo:FIO} and \ref{theo:principal-symbol}. So we recover the classical Chern-Simons theory from large level quantum theory. This is our main result and it has several consequences. 

First we recover some results proved by Andersen in \cite{An}: the quantum representations satisfy the Egorov property. This implies that an element of $\tilde{\Ga}^g_1$ can not act
trivially unless its action on the moduli space is trivial, cf. theorem \ref{theo:faithfulness}. 
Furthermore the quantum representations are asymptotically
unitary for a natural scalar product, meaning that the product of each
operator by its adjoint is the identity modulo a term bounded by the
inverse of the level, cf. theorem \ref{theo:unitarity}. The existence of a scalar product asymptotically invariant was already proved in \cite{An}.

Another important consequence, which is a new result, is the estimate under some transversality assumptions
of the characters of
the mapping class group representation. We prove that these characters
admit an asymptotic expansion and compute explicitly the leading
order term, cf. theorem \ref{theo:char-quant-repr}. It is  a general property of
topological quantum field theory that the character of the
representation of an element
$\varphi$ of
the mapping class group is the three-dimensional invariants of the
mapping torus of $\varphi$. Using path integral methods, Witten
derived  heuristically \cite{Wi} an asymptotically equivalent of the
three-dimensional invariants. The leading order term we find agrees with
 Witten's formula. Let us point out that to our knowledge the
equivalence between the different approaches of topological quantum
field theory remains conjectural. In particular we do not know any
precise statement relating the character of the representation defined
via geometric quantization of the moduli spaces with the invariants of surface bundles defined via
combinatorial-topological methods. 

Our proofs rely on the microlocal techniques which have been developed
in the context of geometric quantization since the seminal work of
Boutet de Monvel and Guillemin \cite{BoGu}. We essentially use results
of our own papers: \cite{oim_qm} for the Fourier integral
operators in this geometric context, \cite{oim_mc} for the
semi-classical properties of connections  and \cite{oim_LF} for the estimate of the trace of a
Fourier integral operator. The role played by the metaplectic
correction in this theory was understood in \cite{oim_mc}. For a part
of the results, this correction is not necessary. As instance we can prove
that the operator in the representation defined with the connection of
\cite{Hi} are Fourier integral operators and that their characters
admit an asymptotic expansion. But for the computations of the principal
symbol of these operators and of the leading order term of the asymptotic
expansion of the characters, we use the metaplectic correction.

The semi-classical limit of the quantum Chern-Simons theory is
essentially conjectural, because it relies on the path integral
formalism. Nevertheless some rigorous results have been obtained, cf. 
\cite{MaNa} for the estimate of the trace norm of the curve operators,
\cite{FrWaWa} and \cite{An} for the asymptotic faithfulness. In particular,
the paper \cite{An} was the first using microlocal
techniques in this context, cf also \cite{An3}. Many papers were devoted to the Witten's formulas for the asymptotic of the
three-dimensional invariant of Seifert manifolds, cf. \cite{BeWi} for
a recent account and references therein. The case of the mapping torus
of finite order diffeomorphisms has been considered in \cite{An2}.

In a companion paper \cite{oim_MG}, we
prove similar results in genus 1 for the quantum representations
of the 
modular group. In particular we derive the asymptotic for the
the torus bundle invariants proved in \cite{Je}. 

The paper is organized as follows. Section \ref{sec:quant-moduli-spac}
is mainly expository. We
introduce the various data necessary to quantize the moduli spaces
which are the complex structures, the Chern-Simons bundle and half-form bundle. We also
discuss carefully the action of the modular group. Since we work
with one marked point, it is necessary to consider various extensions
of the usual mapping class group. 
In section \ref{sec:results}, we
state the results previously mentioned in the
introduction. In section \ref{sec:hitchins-connection} we review the construction of the Hitchin's
connection. We give a general presentation in order to compare the connection
of \cite{Hi} with the one of \cite{AnGaLa}. This makes clear that the
connection with the metaplectic
correction is projectively flat. In sections
\ref{sec:an-algebra-toeplitz} and \ref{sec:asymptotic-flatness}, we prove the asymptotic
flatness. The material in these sections is completely elementary. In
the last sections \ref{sec:semi-class-conn} and
\ref{sec:parallel-transport}, we introduce a general class of
semi-classical connections and prove that their parallel transport are
Fourier integral operators. 

\paragraph{Acknowledgements ---} I am grateful to Richard Wentworth
for helpful discussions and to Chris Woodward for
answering my questions about his work.

\section{Quantization of moduli spaces} \label{sec:quant-moduli-spac}

\subsection{Moduli space and the Chern-Simons bundle} \label{sec:moduli-space-chern-simons}

Let $\Si$ be a compact connected oriented surface without
boundary of genus $g
\geqslant 2$  with a marked point $p$. Let $n$, $d$ be two coprime integers with
$n \geqslant 2$. 
Consider the moduli space $\mo $ of flat
$\SU (n)$-principal bundles over $\Si \setminus \{ p \} $ with holonomy around $p$ equal to $\exp (
2 i \pi d/ n) \id$.
Since $(n,d)=1$, $\mo$ is a smooth compact
manifold. 

For any $[P] \in \mo$, the bundle associated to $P$ via the adjoint
representation is a flat real vector bundle over $\Si \setminus \{
p \}$ whose holonomy around $p$ is trivial. So this associated bundle
is the restriction of a flat real vector bundle $\Ad P$ over $\Si$,
unique up to isomorphism.   
The tangent space of the moduli space at $[P]$ is  
$$ T_{[P]} \mo \simeq H^{1}( \Si, \Ad P).$$ 
The bundle $\Ad P$ has a natural metric coming from  the basic scalar product of
${\mathfrak{su}}(n)$
$$ a \cdot b = -\frac{1}{4 \pi^2} \operatorname{tr} (a b), \qquad a ,
b \in \mathfrak{su}(n) $$
Atiyah and Bott \cite{AB} introduced a symplectic form $\om_\mo$ on
$\mo$. It  
is given by 
$$ \om_\mo ( [a], [b] ) =2 \pi \int_{\Si}  a
\cdot b  $$
where $a$ and $b$ are any closed forms of $\Om^1 ( \Si, \Ad P)$. 

The following facts are proved in \cite{AB}, \cite{Ra}, \cite{MeWo2}: $\mo$ is simply connected, it has no torsion, its second Betti number is one and $n \om_\mo$ is a generator of $H^2 (\mo , \Z) \subset H^2(\mo, \R)$. So there exists a Hermitian line bundle $$ L_{CS} \rightarrow \mo$$ equipped with a connection
of curvature $\frac{n}{i} \om_\mo$. We call $L_{CS}$ the Chern-Simons bundle, it is unique up to isomorphism. 

Various explicit constructions of $L_{CS}$ as a quotient in gauge theory were given in \cite{Fr}, \cite{MeWo} and \cite{BiLa} extending the construction of Ramadas-Singer-Weitsman \cite{RaSiWe} in the case $\Si$
has no marked point. This is important for our purpose because one can deduce that the mapping class group action on $\mo$ lifts to $L_{CS}$, cf. next section. We can also construct by gauge theory an orbifold prequantum bundle on $\mo$ with curvature $\frac{1}{i} \om_\mo$. The $n$th power of this bundle is $L_{CS}$. There is no restriction to consider only $L_{CS}$ because only the $k$-th powers of the orbifold prequantum bundle with $k$ divisible by $n$ admits non trivial global sections. 

\subsection{Mapping class groups}

Since we consider a surface with a marked point $p$, we may
introduce three different mapping class groups 
$$ \Ga_g := \pi_0 (\Diff ^+ ( \Si)) , \quad \Ga_{g,1} := \pi_0 (
\Diff^+ ( \Si, p)) , \quad \Ga_g^1 := \pi_0 ( \Diff^+ ( \Si, D)) $$
where $D$ is an embedded disk in $\Si$ whose interior contains
$p$. Recall that $\Ga_{g,1}$ is an extension of $\Ga_g$ by the
fundamental group of $\Si$,
$$ 1 \rightarrow \pi_1 ( \Si) \rightarrow \Ga_{g,1} \rightarrow \Ga_g
\rightarrow 1.$$
Furthermore, 
$$ 1 \rightarrow \Z \rightarrow \Ga_g^1 \rightarrow \Ga_{g,1} \rightarrow
1$$
where the kernel is generated by a Dehn twist on a loop around $p$. 

The following facts are explained in detail in the paper \cite{oimCS}. First, the mapping class group
$\Ga_{g,1}$ acts on $\mo$; this action does not factor through an action of
$\Ga_g$. Second, the mapping class group $\Ga_g^1$ acts the Chern-Simons bundle by automorphisms of
prequantum bundles. This action lifts the previous action of
$\Ga_{g,1}$ on $\mo$. Nevertheless it does not in general factor through an action
of $\Ga_{g,1}$. Indeed, a Dehn twist on a loop
around $p$ acts on $L_{CS}$ by multiplication by $\exp \bigl( i \pi (n-1)
d^2\bigr)$ in each fiber.

\subsection{Complex structure} 
Suppose $\Si$ is endowed a complex structure compatible with the
orientation. Then by Hodge decomposition, for any $[P] \in \mo$
$$ H^{1}(M, \Ad P) \otimes \C = H^{0,1} ( \Si , (\Ad P) \otimes \C )
\oplus H^{1,0} ( \Si , (\Ad P) \otimes \C )$$
$\mo$ has a complex structure such that the holomorphic
tangent space at $[P]$ is the first summand in the previous
decomposition. This complex structure is integrable, compatible with
$\om_{\mo}$ and positive. So it makes $\mo$ a K{\"a}hler manifold. It may also
be defined by identifying $\mo$ with the moduli
space of holomorphic vector bundles of rank $n$, degree $d$ with a
fixed determinant through the Narasimhan-Seshadri theorem. 
The Chern-Simons bundle has a unique holomorphic
structure such that its $\bar{\partial}$-operator is the $(0,1)$-part of the
connection.

Let $\Ac$ be the space of complex structures of $\Si$ and $$\Teich := \Ac /
\Diff_0^+ (\Si)$$ be the Teichm{\"u}ller space. The mapping class group
$\Ga_g$ acts on $\Teich$. As a fact, the complex structure
of $\mo$ induced by $j \in \Ac$ only depends on the class of $j$ in the
Teichm{\"u}ller space.
Furthermore the actions of $\Ga_{g,1}$ on $\mo$ and
$\Ga_g$ on $\Teich$ are compatible in the sense that for any $\ga \in
\Ga_{g,1}$ and $\si \in \Teich$, the action of $\ga$ on $\mo$ sends
the complex structure induced by $\si$ into the one induced by $\ga (
\si)$. 

For any class $\si$ in the Teichm{\"u}ller space,  we let $\mo_\si$ be $\mo$ endowed with the complex structure induced by $\si$ and $L_\si
\rightarrow \mo_\si$
be the Chern-Simons bundle with the corresponding holomorphic
structure. 

We shall consider $\Teich \times \mo \rightarrow \mo$ as a
smooth family
of complex manifold, so we identify $\mo_\si$ with the slice $\{ \si
\} \times \mo$.  
Denote by $L$ the pull-back of $L_{CS}$ by the
projection from
$\Teich \times \mo$ onto $\mo$. The action of $\Ga_g^1$ on $L$ lifts
the diagonal action of $\Ga_{g,1}$ on $\Teich \times \mo$. 

\subsection{Half-form bundles} 
Let $K$ be the line
bundle over $\Teich \times \mo$ whose restriction to $\mo_\si$ is the canonical bundle $K_\si$ of $\mo_\si$. 
By \cite{Ra}, the canonical class of $\mo$ has a square root. So there exists a line bundle $\delta$  over $\Teich \times \mo$ such that $\delta^2 \simeq
K$. Fix an isomorphism  $\varphi$ 
from $\delta^2$ to $K$. Since $\Teich$ is contractible and $\mo$ simply connected, $\delta$ and
$\varphi$ are unique up to isomorphism. 

The diagonal action of the mapping class group
$\Ga_{g}^1$  on $\Teich \times \mo$  lifts naturally to $K$. Denote
by $\ga_K$ the isomorphism of $K$ corresponding to the action of $\ga
\in \Ga^1_g$. Then define the subgroup of $\Ga^1_g \times \operatorname{Aut} ( \delta)$ 
$$\tilde \Ga ^1_g  =  \{ ( \ga , \ga_\delta) /  \varphi \circ
\ga_{\delta} ^{\otimes 2 }  = \ga_{K} \circ  \varphi \} $$ 
This group is an extension by $\Z/ 2 \Z$ of $\Ga_g^1$. This follows again from the fact that
$\Teich$ is contractible and $\mo$ simply connected. 
Since $\Ga^1_g$
acts on  the Chern-Simons bundle, the same holds for its cover $\tilde \Ga ^1_g$. This is
the reason why we consider here $\Ga^1_g$ instead of $\Ga_{g,1}$.  

For any $\si \in \Teich$, we denote by  $\delta_\si$ the restriction of
$\delta$ to $\mo_{\si}$. This line bundle has a holomorphic structure
and a Hermitian metric determined by the condition that the restriction
of $\varphi$ is an isomorphism $\delta_\si^2 \rightarrow K_\si$ of
holomorphic Hermitian bundles.

\subsection{Quantization} 

Define for any positive integer $k$ and any $\si \in \Teich$, 
$$ \Hilb_{k, \si} := H^0 ( \mo_\si, L_{\si}^k \otimes \delta_\si )$$
This vector space is finite dimensional and has a natural scalar product
defined by integrating the pointwise scalar product of sections against
the Liouville measure.  

Since the canonical class of $\mo$ is negative, it follows from Riemann-Roch-Hilbert theorem and Kodaira vanishing
theorem that the dimension of $\Hilb_{k, \si}$  does not depend on
$\si$. So by elliptic regularity, the $\Hilb_{k, \si}$'s  are the
fibers of a smooth vector bundle $\Hilb_k$ 
with base $\Teich$. Furthermore, the actions of $ \tilde{\Ga}^1_g $ on
the Chern-Simons bundle and $ \delta $ induce an action of the same group on $\Hilb_k$.  

By \cite{AnGaLa}, the bundle $\Hilb_k$ has a natural $\tilde{\Ga}_g^1$-equivariant connection $\nabla^{\Hilb_k}$. It is given by an explicit
for\-mu\-la, cf. theorem \ref{theo:Hitchin}. The paper
\cite{AnGaLa}  does not address the issue of the pro\-jec\-tive
flatness. Actually this projective flatness follows from the results of
\cite{Hi} as it is explained in section \ref{sec:appl-moduli-space}. 

\subsection{Comparison with the usual construction} 

By \cite{MeWo}, the canonical class of $\mo$ is $-2 c_1 (L_{CS})$. So
the fiber bundles $\delta$ and $L^{-1}$ are isomorphic. Since the
Jacobian variety of each $\mo_\si$ is trivial, one can even choose an
isomorphism between $\delta$ and $L^{-1}$ restricting on each slice $\mo_{\si}$ to an isomorphism of
holomorphic vector bundles. So 
\begin{gather} \label{eq:isom}
 \Hilb_{k,\si} \simeq H^0 (M_\si, L_\si^{k-1})
\end{gather}
In \cite{Hi}, Hitchin defined a flat connection on the projective
space bundle over the Teichm{\"u}ller space with fibers ${\mathbb{P}}(H^0
(\mo_\si, L_\si^{k-1})).$ We prove in section \ref{sec:appl-moduli-space}
that this connection is the same as the one induced by
$\nabla^{\Hilb_k}$ via the isomorphism (\ref{eq:isom}).

The bundles $\delta$ and $L^{-1}$ have no reason to be
$\tilde{\Ga}^1_g$-equivariant. However, choosing as above an isomorphism compatible with
the holomorphic structure in each slice, we obtain an equivariant
isomorphism between the projective space bundles. This follows from
the trivial fact that the group of automorphisms lifting the identity of a holomorphic line bundle
with a connected base is $\C^*$. So the projective representation
induced by $\nabla^{\Hilb_k}$ is the same as the one induced by the
Hitchin connection. Observe that this projective representation factors through a
projective representation of $\Ga_{g,1}$. 

\subsection{A conjecture on $\tilde{\Ga}^g_1$}

Recall that the Hodge line bundle is the complex line bundle over the
Teichm{\"u}ller space whose fiber at $\si$ is $
\wedge^{\operatorname{top}} H^1(\Si_\si, {\mathcal{O}} )$. Here
$\Si_\si$ is the surface $\Si$ endowed with the complex structure $\si$
and $ {\mathcal{O}}$ is the structural sheaf. Denote by $\lambda$ the
pull-back of the Hodge line bundle by the projection $\mo \times
\Teich \rightarrow \Teich$. There is an obvious action of $\Ga_{g}$
on $\lambda$. Furthermore the previously defined actions of $\Ga_g^1$
on $K$ and $L^2_{CS}$ factor through actions of $\Ga_{g,1}$. We conjecture
that there exists a $\Ga_{g,1}$-equivariant isomorphism
\begin{gather} \label{eq:conj} 
 K \simeq L_{CS}^{- 2} \otimes \lambda^{n^2 -1}
\end{gather}
A similar result is proved for the moduli space without marked
point in \cite{Fu}. 

Introduce a symplectic basis of the first homology group of $\Si$. So the action on the homology defines a morphism $\Psi: \Ga_{g} \rightarrow \operatorname{Sp}(2g, \R)$. Let $\operatorname{Mp}(2g)$ be the metaplectic group and $q$ be the projection $\operatorname{Mp}(2g) \rightarrow \operatorname{Sp}(2g,\R)$. If there exists a $\Ga_{g,1}$-equivariant isomorphism (\ref{eq:conj}), then we have the following group isomorphisms
$$\tilde{\Ga}_g^1 \simeq \begin{cases} \Ga^1_g
\times \Z_2  \text{ if $n$ is even}\\ 
\{ (\gamma , h) \in \Ga ^1_g \times \operatorname{Mp}(2g)/ \; \Psi (\gamma) = q (h) \} \text{ if $n$ is odd.}
\end{cases}
$$

\section{Semi-classical results} \label{sec:results}

\subsection{Asymptotic flatness} 

The sequence of connections  $( \nabla^{\Hilb_k} ) $ is asymptotically flat.

\begin{theo} \label{theo:curvature}
For any  vector fields $X$, $Y$ of $\Teich$, any compact set $K$ of
$\Teich$, there exists $C>0$, such that the
curvature at any $\si \in K$
$$ R^{\Hilb_k}(X, Y) (\si ) = [ \nabla_X ^{\Hilb_k}, \nabla_Y ^{\Hilb_k} ] - \nabla ^{\Hilb_k}_{[X , Y ]} : \Hilb_{k, \si} \rightarrow \Hilb_{k ,\si} $$    has a uniform norm bounded by $C k^{-1}$. 
\end{theo}
We will propose two different proofs. The first in section
\ref{sec:asymptotic-flatness} is completely elementary. The second one
relies on a much more general result, theorem
\ref{theo:curvature_semi_classique}. In both cases, the result follows from miraculous cancellations, happening only with the metaplectic correction and which were first observed in \cite{oim_mc}.

Let $\si \in \Teich$. Let $h = ( \ga , \ga_{\delta}) \in \tilde \Ga_g^1$ and $p$ be a path
of the Teichm{\"u}ller space from $\ga.\si$ to $\si$. For any integer
$k$, the action of $h$ restricts to a linear map from
$\Hilb_{k, \si}$ to $\Hilb_{k, \ga.\si}$. By composing this map with the
parallel transport along $p$, we obtain an endomorphism of $\Hilb_{k,
  \si}$ that we denote by $U_k ( h,p)$.
It follows from theorem \ref{theo:curvature} that these maps are defined
independently of the choice of $p$ up to a $O(k^{-1})$ term, that is
for any paths $p_1$ and $p_2$ from $\ga.\si$ to $\si$,      
$$ U_k(h,p_1) =  \la_k(p_1,p_2)  U_k(h,p_2) $$
where $\la_k(p_1,p_2)$ is a sequence of complex numbers equal to $1 +
O(k^{-1})$ . 
Furthermore we obtain an asymptotic representation of $\tilde \Ga_g^1$
in the sense that 
$$ U_k(h,p). U_k(h',p') = \mu_k(p,p')  U_k(hh', p''),  $$
with $\mu_k(p,p')$  a sequence of complex numbers equal to $1 +
O(k^{-1})$.    
\subsection{Semi-classical properties of the quantum representation}
Our main result says that the sequences $(U_k( h,p))_k$ are Fourier
integral operators. This means that the Schwartz kernel of such a
sequence concentrate in a precise way on a Lagrangian submanifold of
$\mo \times \mo^{-}$. Furthermore the restriction of the Schwartz kernel to
this canonical relation is described in first approximation by a flat section of the prequantum
bundle $L_{CS} \boxtimes L_{CS}^{-1}$. Here, the Lagrangian manifold
and the flat section are the graphs of the action of
the mapping class $\ga\in \Ga_g^1$ on the moduli
space $\mo$ and the Chern-Simons bundle respectively.

To state the result, let us recall our convention for Schwartz kernels. Since $\Hilb_{k,\si}$ is a finite dimensional Hilbert space, we have
an isomorphism $\End (\Hilb_{k,\si}) \simeq \Hilb_{k,\si} \otimes
\overline{\Hilb}_{k,\si}$. The latter space identifies with the space
of holomorphic sections of the bundle $$(L^k_\si \otimes \delta_{\si}) \boxtimes
(\overline{L}^k_\si \otimes \overline{\delta}_{\si}) \rightarrow \mo_\si \times \overline
\mo_\si $$
We call the section associated to an endomorphism its Schwartz kernel.  

\begin{theo} \label{theo:FIO}
Let $h = ( \ga , \ga_\delta) \in \tilde \Ga_g^1$ and $p$ be a path of the Teichm{\"u}ller space from $\ga.\si$
to $\si$. Then the Schwartz kernel of $U_k(h,p)$ has the following form 
$$  \Bigl( \frac{k}{2 \pi } \Bigr)^{m} F^k(x,y) \otimes f (x,y , k) + O( k^{-\infty})$$ 
where $m$ is the complex dimension of $\mo$ and 
\begin{itemize}
\item $F$ is a section of $L_\si \boxtimes \overline {L}_\si$ such that $| F(x,y)
| <1 $ if $x \neq \ga.y $, $$ F (\ga .y,y) = (\ga .u) \otimes \bar{u}$$
for all $y \in \mo$ and $u \in L_y$ of norm 1, and $ \bar{\partial} F \equiv 0 $
modulo a section vanishing to any order along $\{ ( \ga .y, y) / \;
y \in \mo\}$.
\item
  $f(.,k)$ is a sequence of sections of  $ \delta_{\si}  \boxtimes  \bar{\delta}_\si
\rightarrow \mo_\si \times \overline \mo_\si $ which admits an asymptotic expansion in the $\Ci$
  topology of the form 
$$f_0 + k^{-1} f_1 + k^{-2} f_2 + ...$$
whose coefficients satisfy $\bar{\partial} f_i \equiv 0  $
modulo a section vanishing to any order along $\{ ( \ga .y, y) / \;
y \in \mo\}$.
\end{itemize}
\end{theo}

Let $(F_k \rightarrow N)_k$ be a family of Hermitian bundle over the
same base and $(s_k \in \Ga ( N, F_k))_k$ be a family of section. Then we
say that $s_k$ is $O(k^{-\infty})$ if for any integer $N$, there
exists $C_N >0$ such that
$$   \| s_k(x) \| \leqslant C_N k^{-N} , \qquad \forall x \in N .$$
The previous theorem implies in particular that the sequence of
Schwartz kernel of $U_k(h, p)$ is $O(k^{-\infty})$ over any compact
set of $\mo^2 \setminus \{ ( \ga.x, x ) , x \in \mo \}$. 
To complete the description of the leading order term, we determine
the restriction of $f_0$ to $\{ ( \ga .y, y) / \;
y \in \mo\}$. It depends on the morphism $\ga_\delta$ of the half-form
bundle. 

For any $x$ in $\mo$ and $\si $ in $\Teich$, let $E_{\si,x}$ be
the subspace of $T_x \mo_\si \otimes \C$ which consists of the vectors of
type $(1,0)$. For any $\si, \si' \in \Teich$, let $\pi_{\si' , \si, x}$ be the
projection from $E_{\si' , x}$ onto $E_{\si, x }$ with kernel
$\bar{E}_{\si' , x} $. Recall that by definition of the half form
bundle, there is a preferred isomorphism $\varphi$ from $\delta^2$ to $K = \det
E^*$. Since the moduli space and the
Teichm{\"u}ller space are simply-connected, there exists a unique family $\Psi_{\si
  ,\si' , x}: \delta_{\si,x} \rightarrow \delta_{\si',x}$ of isomorphisms
satisfying 
$$  \varphi_{\si', x} \circ \Psi_{\si, \si' , x} ^2 = \pi_{\si ' , \si
  ,x }^* \circ  \varphi_{\si , x }, $$  
depending continuously on $\si$, $\si'$ and $x$ and such that
$\Psi_{\si, \si , x}$ is the identity of $\delta_{\si, x}$ for any
$\si$. 

\begin{theo} \label{theo:principal-symbol}
Let $h = ( \ga , \ga_\delta) \in \tilde \Ga_g^1$ and $p$ be a path of the Teichm{\"u}ller space from $\ga.\si$
to $\si$.   
Then for any $x \in \mo$, the Schwartz kernel of
$U_k(h,p)$ at $( \ga.x,x)$ is equal to 
$$  \Bigl( \frac{k}{2 \pi } \Bigr)^{m} \bigl[ (\ga.u)  \otimes
\bar u  \bigr]^k
\otimes \bigl[ \Psi_{\ga .\si, \si , x} (\ga_\delta(v))  \otimes \bar{v} \bigr]
  + O( k^{m -1 })$$ 
where $m$ is the complex dimension of $\mo$ and $u$, $v$ are any
vectors with norm 1 of $L_x$ and $\delta_{\si, x}$
respectively.
\end{theo}

Theorems \ref{theo:FIO} and \ref{theo:principal-symbol} are
consequences of theorem \ref{theo:paral_transp}, the latter being a
 generalization of theorem 7.1 in \cite{oim_mc}. We provide a
complete proof in section \ref{sec:parallel-transport}.

\subsection{Asymptotic faithfulness and asymptotic unitarity} 

Let us explain the relation with Andersen's work. By the Egorov theorem for Toeplitz operator (theorem 3.3 of \cite{oim_mc}), one deduce from theorems \ref{theo:FIO} and \ref{theo:principal-symbol} the following fact: if  $(T_k)_k$ is a Toeplitz operator of $( \Hilb_{k, \si})_k$ with symbol $f \in \Ci ( \mo)$, then $$(U_k(\ga, p)^{-1} T_k U_k ( \ga,
p))_k$$ is a Toeplitz operator with principal symbol $\ga_{\mo}^* f$. This fact was proved in \cite{An}. Observe that the Egorov property depends only on the projective class of $U_k(\ga, p)$ unlike theorems \ref{theo:FIO} and \ref{theo:principal-symbol}.

The second part of the following theorem was the main argument to prove the asymptotic
faithfulness of the quantum representation in \cite{An}.

\begin{theo} \label{theo:faithfulness}
Let $h = ( \ga , \ga_\delta) \in \tilde \Ga_g^1$ and $p$ be a path of the Teichm{\"u}ller space from $\ga.\si$
to $\si$. If $U_k( \ga, p) = \id $ when $k$ is sufficiently large,  then $\ga$ acts trivially on the Chern-Simons
bundle and $\ga_{\delta} = \id$. If 
$$  U_k( \ga, p) T U_k( \ga, p)^{-1} = T , \qquad \forall T \in  \End ( \Hilb_{k,\si})    $$ 
when $k$ is sufficiently large, then $\ga$ acts trivially on the
moduli space $\mo$.
\end{theo}
 The first part is an immediate consequence of theorems \ref{theo:FIO} and
\ref{theo:principal-symbol}. The second part follows from Egorov property using the fact that any smooth function of $\mo$ is the symbol of a Toeplitz operator.

Another corollary of theorems \ref{theo:FIO} and
\ref{theo:principal-symbol} is the asymptotic unitarity.

\begin{theo} \label{theo:unitarity}
Let $h = ( \ga , \ga_\delta) \in \tilde \Ga_g^1$ and $p$ be a path of the Teichm{\"u}ller space from $\ga.\si$
to $\si$. Then 
$$ U_k( \ga, p ) U_k( \ga , p)^* = \id + O( k^{-1})$$
where the $O(k^{-1})$ is for the uniform norm of operators of $\Hilb_{k,\si}$.
\end{theo}   
This is a general properties of Fourier integral operators whose
symbols are half-form morphisms, cf. \cite{oim_mc} and \cite{oim_LF}. In proposition 2 of \cite{An}, Andersen introduced a metric asymptotically preserved by the quantum representations which is likely to be the same as the one in this paper.

\subsection{Characters of the quantum representation}

Let us turn to the asymptotic of the characters of the quantum
representations. In \cite{oim_LF}, we proved that the trace of a
Fourier integral operator whose canonical relation intersects
transversally the diagonal has an asymptotic expansion. We also computed
explicitly the leading order term. By theorems \ref{theo:FIO} and
\ref{theo:principal-symbol}, these results can be applied to the
mapping class group representations.

\begin{theo} \label{theo:char-quant-repr}
Let $h = ( \ga , \ga_\delta) \in \tilde \Ga_g^1$ and $p$ be a path of the Teichm{\"u}ller space from $\ga.\si$
to $\si$.  Assume that the fixed
  points of the action
  of $\ga$ on $\mo$ are all non-degenerate. Then we have   
$$ \operatorname{Tr} (U_k ( \ga, p )) = \sum_{x\in \mo / \ga.x =x } \frac{ i^{m(
    \ga_\delta ,x )} . u ( \ga,x)^k }{| \det ( \id - L(\ga, x) )|^{1/2}} + O(k^{-1})
$$
where for any fixed point $x \in \mo$,
\begin{itemize}
\item  $L(
\ga, x)$ is the linear tangent map at $x$ of the action of
$\ga$ on $\mo$,
\item the action of $\ga$ on the fiber of the Cherns-Simons bundle at
  $x$ is the multiplication by the complex number $u(\ga , x )$, 
\item  $m(
     \ga_\delta , x ) \in \Z/4\Z$ is the index of the automorphism
     $\Psi_{\ga .\si, \si , x} \circ \ga_\delta $ of $\delta_{\si,x}$.
  \end{itemize}
\end{theo}

The indices $m(\ga_\delta , x )$ are defined in $\cite{oim_LF}$.  Let
us discuss the other terms of the formula. 
Let $\Phi$ be a diffeomorphism of $\Si$ fixing $p$. Let $[P] \in \mo$
be such that the restriction of $\Phi$ to $\Si \setminus \{ p\}$ lifts
to an isomorphism $\Phi_P$ of $P$.
Then  the
induced isomorphism of $ \Ad P \rightarrow \Si \setminus \{ p \}$ extends by
continuity at $p$ to an isomorphism $\Ad \Phi_{P}$ of $\Ad P$. We have 
$$ L(\ga, x) = ( \Ad \Phi_P)_* : H^{1}( \Si, \Ad P) \rightarrow H^1 (
\Si, \Ad P) $$
if $\ga$ is the mapping class of $\Phi$ and $x = [P]$. 
 Introduce the mapping torus 
$$ ( \Ad P)_{\Ad \Phi_P} = ( \Ad P \times \R ) / ((\Ad \Phi_P ) (x),
t) \sim ( x, t+1)$$
It is a flat bundle over the mapping torus of $\Phi$. Is is a well
known fact that the Reidemeister torsion of $( \aad P)_{\Ad \Phi_P}$ satisfies
$$ | \tau \bigl(  ( \aad P)_{\Ad \Phi_P} \bigr)| = \frac{1}{| \det ( \id - L(\ga, x)
  )|} $$ 
We refer the reader to \cite{Je:thesis} for a proof. 
The complex numbers $u( \ga, x)$ are related to the Chern-Simons
invariant of the mapping tore of $\Phi_P$. This will be explained in another paper, the difficulty is that
the base of this bundle is not a closed manifold, so we must be
careful in the definition of the Chern-Simons invariant.

\section{Hitchin's connection}  \label{sec:hitchins-connection}

\subsection{Holomorphic differential operators} 

Let $M$ be a complex manifold and $F\rightarrow M$ be a holomorphic line bundle. Consider the algebra of differential operators
acting on $\Ga ( M, F)$. It is the direct sum of the subalgebra of
holomorphic differential operators and the left-ideal $\I$ generated by the
anti-holomorphic derivations. More explicitly, introduce a local
holomorphic trivialization of $F$ and a system  $(z^i)$ of holomorphic
coordinates of $M$. Then each differential operator is of the form 
$$ \sum_{\al \in \Z^n} a_\al
\partial^{\al(1)}_{z^1}... \partial^{\al(n)}_{z^n}  + \sum_{\al, \be
  \in \Z^n, \be \neq 0} a_{\al, \be}
\partial^{\al(1)}_{z^1}... \partial^{\al(n)}_{z^n}
\partial^{\be(1)}_{\bar z^1}... \partial^{\be(n)}_{\bar z^n} ,$$
where the coefficients $a_\al$  and $a_{\al, \be}$ are smooth
functions. The first summand is a holomorphic differential operator
and the second one belongs to the ideal $\I$. 

We denote by $\difh_k (F)$ the bundle whose sections are the
holomorphic differential operators of order $k$ acting on $\Ga (
M,F)$. $\difh_k (F)$  has a natural holomorphic structure, such that
its holomorphic
sections are the holomorphic differential operators with holomorphic
coefficients. Observe that for any holomorphic
differential operator $P$ and smooth section $Z$ of $T^{1,0} M$, one
has
\begin{gather} \label{eq:dbar_op_dif}
 D_{\bar{Z}} P =  [ D_{\bar{Z}} , P] \mod \I
\end{gather}
where $ D_{\bar{Z}}$ denote the derivative of sections of $\difh_k
(F)$ (resp. $F$) on the left hand side (resp. right hand side). 
  
\subsection{Variations of complex structures} \label{sec:vari-compl-struct}

Let $U$ and $M$ be two manifolds. Consider a smooth family $(j_u)_{u
\in U}$ of complex structures of $M$. Denote by $M_u$ the complex manifold $(\{ u \}
\times M, j_u)$. Let $E$ be the complex vector bundle
over $U \times M$ with fibers $$E_{u,x} = T^{1,0}_x M_u.$$ 
We call $E$ the relative holomorphic tangent bundle. We shall
often consider the decomposition of the tangent space of
$U \times M$ given by 
$$ T_{u,x} (U \times M ) \otimes \C  = (T_{u} U \otimes \C ) \oplus
E_{u,x} \oplus \bar{E}_{u,x} $$
Let $X$ be a vector field of $U$. Since $j_u^2 = - \id$, the
derivative of $(j_u)$ with respect to $X$ has the form 
$$ X. j = \mu (X) + \bar  \mu (X) $$
where $\mu$ is a section of $\Hom ( \bar {E}, E)$. Let $Z$ be a smooth
section of $E$. Consider $Z$ and $X$ as vector fields of $U \times
M$. Observe that the Lie bracket $[X,\bar Z]$ is tangent to $E
\oplus \bar E$. Furthermore,  
\begin{gather} \label{eq:der_struct_comp}
[X, \bar {Z} ] = \tfrac{i}{2} \mu (X) ( \bar {Z} ) \mod \bar {E}
.
\end{gather}
which follows easily from the fact that the projection from $E\oplus
\bar E$ onto $\bar E$ is $\frac{1}{2} ( \id + i j )$. 

\subsection{Connection} \label{sec:connection}

Consider as previously two manifolds $U$ and $M$ with a smooth
family $(j_u)_{ u \in U}$ of complex structures. Let $(F_u \rightarrow
M_u)_{u \in U}$ be a smooth family of holomorphic line 
bundles. Assume that $M$ is compact and that the dimension
of the space $H^{0}(M_u, F_u)$ of holomorphic sections does not depend on
$u$. Then by elliptic regularity, there exists a smooth vector bundle
$\Hilb$ with base $U$ and
fibers $H^{0} (M_u, F_u)$,  such that its smooth sections are the smooth
families of holomorphic sections (cf \cite{BeGeVe}, chapter 9.2).  

Let $F$ and $\difh_2 (F)$ be the bundles over $U\times M$ which
restrict over any slice $M_u$ to  $F_u $ and $\difh_2(F_u)$
respectively. If $s$ is a section of $F$, we denote by $\bar{\partial}
s$ the section of $F \otimes \bar E ^*$ which restricts over $M_u$ to
$\bar \partial_u s_u$. So a smooth section of $\Hilb$ is by definition a
smooth section of $F$ satisfying $\bar \partial s = 0$. We use the same
notation with $\difh_2 (F)$ and more generally with any family of holomorphic
bundles.

Introduce a connection $\nabla$ on
$F$, such that its restriction to any $M_u$ is compatible
with the holomorphic structure of $F_u$. Introduce a section  $P$ of
$\difh_2 (F) \otimes p^*(T^* U)$ where $p$ is the projection from
$U \times M$ to $U$. 
We would like to define a connection
on $\Hilb \rightarrow U$ whose covariant derivative in the direction
of $X \in \Ga (U, TU)$ is given by
\begin{gather} \label{eq:connection} 
 \nabla_X +  P(X) .
\end{gather}
The following lemma provides a sufficient condition for the
connection to be well-defined. We denote by $R^{\nabla}$ the curvature of
$\nabla$. 

\begin{prop} \label{prop:conn_bien_def}
Assume that for any section $Z$ of $E$, we have 
$$ [\bar \partial P(X)](\bar{Z})  = \frac{i}{2} \nabla_{\mu (X) (\bar{Z} )} + R^{\nabla}
(X,\bar Z) $$
Then if $s$ is a section of $F$ whose restriction to each $M_u$ is
holomorphic, the same holds for $(\nabla_X +  P(X)) s$.
\end{prop}

\begin{proof} 
We show that the assumption implies that for any smooth section $s$
of $F$ and any point $u$ of $U$  
$$\bigl( [\nabla_{\bar Z}, \nabla_X +
P(X) ] s \bigr)_u = Q_u s_u$$  
with $Q_u$ a differential operator of $F_u$ in the ideal $\I_u$ generated
by the anti-holomorphic derivations. 
By (\ref{eq:der_struct_comp}), we have that
$$ [ \nabla_{\bar Z} , \nabla_X ] = - \tfrac{i}{2} \nabla_{\mu (X) (
  \bar{Z}) } - R^\nabla ( X, \bar Z) \mod \I $$
By (\ref{eq:dbar_op_dif}), we have that 
$$ [ \nabla_{\bar Z} , P(X) ] = [\bar \partial P(X) ] ( \bar Z) \mod \I.
$$
The conclusion follows.
\end{proof}

\subsection{Unicity}\label{sec:unicity}
Let us discuss the unicity of a connection of the form
(\ref{eq:connection}) and satisfying the assumption of proposition \ref{prop:conn_bien_def}. Assume that $M$ is connected and that for any $u$, $M_u$ has no
holomorphic vector field. Suppose that $(\nabla, P)$ satisfies the
hypothesis of proposition \ref{prop:conn_bien_def} for any vector field $X$.   Let $( \nabla ' , P')$ be
another pair satisfying the same assumption. Assume that for any $u\in
U$ and $X \in T_uU$,  $P(X)_u$
and $P'(X)_u$ are second-order differential operators with the same
principal symbol. Then there
exists a form $\al \in \Om^1 (U)$ such that
\begin{gather} \label{eq:unicite_con}
 \nabla'_X + P' (X) = \nabla_X + P(X) + \al(X)\id 
\end{gather}
for any vector field $X$ of $U$. 

Indeed, $\nabla$ and $\nabla'$
differ by a one-form $\be \in \Om^{1} ( U \times M , \End F)$ which
vanishes in the directions tangent to $\bar E$. Let $\tilde{\be}$ be the
section of $\End F \otimes p^* T^*U$ such that $\be (X) = \tilde \be
( X)$, for any vector field $X$ of $U$. Then it is easily checked that
the pair $(\nabla' - \be, P' + \tilde \be)$ satisfies the assumption of
proposition \ref{prop:conn_bien_def}. 
Since
$$ \nabla'_X + P' (X) = \nabla'_X - \be(X) +P'(X) + \tilde \be(X) =
\nabla_X + P'(X) + \tilde \be(X)$$
we may assume that $\nabla = \nabla'$. 
Next, the hypothesis of proposition \ref{prop:conn_bien_def} implies that
\begin{gather} \label{eq:bar_op}
\bar{\partial} (P(X) - P' (X) ) = 0
\end{gather}
Since $P(X)_u$ and $P'(X)_u$
have the same principal symbol, $P(X)_u - P'(X)_u$ is a first order
holomorphic differential operator. Since $H^0 (M_u, E_u) = 0$,
equation 
(\ref{eq:bar_op}) implies that $P(X)_u -
P'(X)_u$ is a zero-order holomorphic differential operator. In other
words, $P(X)_u - P'(X)_u$ is the multiplication by a function. By
equation (\ref{eq:bar_op}), this function is holomorphic, hence
constant because $M$ is compact and connected. This proves that $P' = P + \al$ with $\al \in \Om^1 (U)$.  
The desired equation (\ref{eq:unicite_con}) follows. 

\subsection{A preliminary computation}

To apply proposition \ref{prop:conn_bien_def}, we need to compute the
$\bar \partial$ of some second order differential  operator. The
parameter space doesn't enter in the calculation, so we assume in this
subsection that $U = \{ pt \}$. Suppose that the complex manifold $M$
has a K{\"a}hler metric, and that the holomorphic line bundle $F$ has a
Hermitian metric. Hence $F$ and $T^{1,0}M$ have canonical connections
compatible with the metric and the holomorphic structure (Chern connection). Let $G $ be a section
of the second symmetric tensor power $S^2 ( T^{1,0} M)$. Define the
holomorphic differential operator acting on the sections of $F$
\begin{gather} \label{eq:def_delta}
 \Delta ^G  s = \operatorname{Tr} ^{\End (T^{1,0}M) } ( \nabla
 ^{T^{1,0}M \otimes F} ( G \lrcorner  \nabla ^F s ))
\end{gather} 
More explicitly if $\partial_1, \ldots  , \partial_n$ is a local
frame of $T^{1,0}M$ and $\ell_1, \ldots \ell_n$ is the dual frame,
$$ \Delta ^G s = \sum _k \ell_k (  \nabla_k ^{T^{1,0}M \otimes F} (\sum_{i,j}  G_{ij} \partial_i \otimes \nabla_j ^F s ))$$
where $\nabla_i$ is the covariant derivative with respect to
$\partial_i$ and $G = \sum G_{ij} \partial_i \otimes \partial_j$.
 
\begin{prop} \label{prop:prel-comp}
Assume that $G$ is a holomorphic section of $S^2 ( T^{1,0} M)$, then 
$$ \bar \partial \Delta ^G = \sum_{i,j}( 2 R^F + R^{\operatorname{det}} ) ( \cdot , \partial_i) G_{ij} \nabla_j^F  + \theta ^F $$ 
where $R ^F$ is the curvature of $ \nabla ^F$, $R ^{\operatorname{det}}$ is the curvature of the Chern connection of $\wedge ^n T^{1,0} M$ and $\theta ^F$ is the one-form of $M$ given by 
$$  \theta ^F ( \bar Z) = \sum _{k,i,j} \ell_k ( \nabla_k ^{T^{1,0} M} ( R ^F ( \bar {Z} , \partial_i ) G_{ij} \partial_j ))$$
for any (local) holomorphic section $Z$ of $T^{1,0} M$.  
\end{prop}

This is a slight generalization of a computation in \cite{Hi}, page
364. The assumption that the metric is K{\"a}hler is used for certain symmetries of the curvature tensor of $T^{1,0 } M$. 

\subsection{The setting} \label{sec:settting}

Assume now that $M$ is a symplectic manifold and that $(j_u)_{u \in
  U}$ is a family of compatible complex structures. So each $M_u$ is a
K{\"a}hler manifold. 

Let $L_M \rightarrow M$ be a prequantum bundle, that is a Hermitian
line bundle with a connection of curvature $\frac{1}{i} \om$. For any
$u$, denote by $L_u \rightarrow M_u$ the bundle $L_M$ with the
holomorphic structure compatible with the connection and $j_u$. Denote
by $L$ the pull-back of $L_M$ by the projection $U \times M \rightarrow M$ and endow $L$ with the pull-back connection. 

Consider a pair $(\delta, \varphi)$ which consists of a  line bundle $\delta$ over $U \times M$ with an isomorphism $\varphi$ from $\delta
^2$ to $\wedge ^{\top} E^*$. Such a pair exists if and only if the second
Stiefel-Whitney class of $M$ vanishes. The restriction $\delta_u$ of
$\delta$ to $M_u$ has  holomorphic and Hermitian structures determined
by the condition that the isomorphism $\varphi_u : \delta_u^2
\rightarrow \wedge ^{\top} E_u^*$ is an isomorphism of Hermitian
holomorphic bundles. We call $( \delta_u , \varphi_u)_{u \in U}$ a family of
half-form bundles. 

Let us define a connection on $\delta$. First consider the connection
$\nabla^{E \oplus \bar E}$ on $ E \oplus \bar{E}$ such that its
restriction to each slice $M_u$ is the
Levi-Civita connection of the K{\"a}hler metric of $M_u$ and the covariant
derivative in a direction tangent to $U$ is the obvious one. This makes sense because the restriction of $E \oplus \bar{E}$ to $U \times
\{ x\}$ is the trivial bundle with fiber $T_{x}M \otimes \C$. 
Next we consider the following connection on $E$
$$ \nabla ^E  = \pi \circ \nabla^{E \oplus \bar E}, $$ 
where $\pi = \frac{1}{2}(  \id - i j)$ is
the projection of $E \oplus \bar{E}$ onto $E$ with kernel $\bar
E$. This defines a connection on the associated bundle
$\wedge^{\operatorname{top}} E^*$ and finally a connection
$\nabla^\delta$  on $\delta$.

Our aim is to apply the construction of chapter \ref{sec:connection} to the bundle $F= L^k \otimes
\delta$. So consider the vector space  
$$ \Hilb_{k, u } = H ^{0} (M_u , L_u ^k \otimes \delta_u) $$
with the scalar product obtained by integrating the punctual Hermitian
product of sections against the Liouville measure of $M$. For any
compact set $C$ of $U$, if $k$ is sufficiently large, the dimension of
$\Hilb_{k, u}$ is constant when $u$ runs over $C$. Here we do the
global assumption that there exists $k_0$ such that the dimension of
$\Hilb_{k, u}$ does not depend on $u \in U$ when $k \geqslant k_0$. 
For $k \geqslant k_0$, the spaces $\Hilb_{k, u }$ are the fibers of a
smooth vector bundle $\Hilb_k$ over $U$, whose sections are the smooth
families of holomorphic sections. In the sequel we always assume that $k
\geqslant k_0$.

\subsection{Existence of the connection} \label{sec:existence}

Consider the same data as in section \ref{sec:settting}. To define a
connection on the bundle $\Hilb_{k} \rightarrow U$, we assume first that $M$ is simply
connected. Then for any tangent vector $X$ of $U$ at $u$, introduce the section $G(X)$ of $S^{2 } (T^{1,0}M_u)$ such that
\begin{gather} \label{eq:def_G}
  \sum _{i,j}  G_{ij} \om ( \partial_i , \cdot )   \partial_j = \mu
(X)_u, \qquad  G(X)= G_{ij} \partial_i \otimes \partial_j
\end{gather}
Our second assumption is that these sections $G(X)$
are holomorphic.

To apply proposition \ref{prop:conn_bien_def} to the bundle $F= L^k
\otimes \delta$, we need to compute the curvature $R ^{\delta}$ of $\nabla ^{\delta}$. The following result is proved in \cite{AnGaLa}.
\begin{prop} \label{prop:courbure_demi_forme}
For any tangent vector $X \in T_uU$  and section $Z$ of $E$, we have
$$ 4 R^{\delta} (X, \bar{Z} ) = \theta ^{L_u} (\bar Z)$$ 
where $\theta ^{L_u}$ is defined as in proposition \ref{prop:prel-comp}.
\end{prop}

As a consequence of proposition \ref{prop:prel-comp}, we have for any tangent vector $X \in
T_uU$ and section $Z$ of $E_u$, 
\begin{xalignat*}{2} 
 \bigl[ \bar \partial \Delta ^{G(X)} \bigr] (\bar Z)  = & 2i k\nabla
 ^{L^k \otimes \delta} _{\mu (X) (\bar Z)} +  k \theta^{L_u} ( \bar Z)  + \theta
 ^{\delta_u} ( \bar Z)  \\ 
= &  4k \bigl[
\tfrac{i}{2} \nabla ^{L^k \otimes \delta} _{\mu (X) (\bar Z)} +
R^\delta (X, \bar Z) \bigr] + \theta ^{\delta_u} (\bar Z)
\end{xalignat*}
by proposition \ref{prop:courbure_demi_forme}. For $k=0$, this implies that $\bar
\partial \theta ^{\delta_u} = 0$. Since $M$ is simply connected, by
Hodge decomposition, the Dolbeault cohomology group $H ^{1 , 0 }
(M_u)$ vanishes for any $u$. So there exists a function $H (X)$ such
that 
\begin{gather} \label{eq:def_H}
\bar \partial H(X) =  \theta ^{\delta_u}, \qquad \int_M H(X) \;
\om^n = 0
\end{gather}
This defines a section $H \in \Ga ( U \times M , p^* (T^*U))$. Then 
\begin{gather*} 
  \nabla ^{L^k \otimes \delta }_X + \frac{1}{4k} \bigl(\Delta ^{G(X)} - H(X) \bigr)
\end{gather*}  
satisfies the assumption of proposition \ref{prop:conn_bien_def}. 
To summarize, we have prove the following theorem. This was the main theorem of \cite{AnGaLa}.

\begin{theo}\label{theo:Hitchin}  Let $M$ be a simply-connected compact
  symplectic manifold endowed with a
  prequantum bundle $L_M$ and a family $(j_u, \delta_u, \varphi_u)_{u
    \in U}$  of compatible complex structures with half-form bundles.  Assume
  that the sections $G(X)$, $X \in TU$,
  defined in (\ref{eq:def_G}) are holomorphic. Then the bundle
  $\Hilb_k \rightarrow U$ with fibers $\Hilb_{k,u} = H^0(M_u, L^k_u
  \otimes \delta_u)$ admits a connection $\nabla^{\Hilb_k}$ defined by   
\begin{gather} \label{eq:def_connection_Hitchin}
 \nabla^{\Hilb_k} _X: =\nabla ^{L^k \otimes \delta }_X + \frac{1}{4k}
 \bigl(\Delta ^{G(X)} - H(X) \bigr), \qquad X \in \Ga(U, TU) 
\end{gather}  
where $\Delta^{G(X)}$ and $H(X)$ are given by the equations
(\ref{eq:def_delta}) and (\ref{eq:def_H})
\end{theo}

\subsection{Application to the moduli space}\label{sec:appl-moduli-space}

We can apply the previous constructions to the moduli space $\mo$ and
the Chern-Simons bundle, cf. chapter
\ref{sec:moduli-space-chern-simons}, with the Teichm{\"u}ller space as a
parameter space. 
All the assumptions we made are satisfied. We obtain a connection
explicitly given by (\ref{eq:def_connection_Hitchin}). We assert that this connection is projectively
flat. This can be deduced from the results of \cite{Hi} and our
discussion about the unicity.

Indeed, in \cite{Hi}, Hitchin considered the quantum space $\Hilb'_k$ over $\Teich$ whose
fiber at $\si$ is $H^0_\si ( \mo_\si, L_\si^k)$. He proved the existence
of a projectively flat connection of the form 
$$  \nabla_X^{L^k} + P_k'(X), \qquad X \in \Ga(\Teich, T \Teich)     $$ 
where for any $\si$, $P'_k(X)_\si$ is a second order differential
operator. Furthermore the symbol of $P'_k(X)_\si$, viewed as a section of $S^2(E_\si)$, is
$G/ ( 4(k + 1)) $. Because of the isomorphism $\delta \simeq L^{-1}$,
this defines a connection on $\Hilb_{k+1}$. It follows from the
discussion of chapter \ref{sec:unicity} that this connection is the
same as (\ref{eq:def_connection_Hitchin}) up to a term of the form
$\al \id$ with $\al \in \Om^1 ( \Teich )$. So the curvatures of the
two connections differ by $d \al \id$.

\section{An algebra of Toeplitz operators} \label{sec:an-algebra-toeplitz}

Consider a complex compact manifold $M$ and a family of Hermitian holomorphic line bundles
$F=(F_j \rightarrow M)_{ j \in J}$. For any $j \in J$, let $\op ( F_j ) $ be the
algebra of  holomorphic differential operators acting on the sections
of $F_j$. Consider the subalgebra $\op _{sc} (F) $ of $\prod _{j \in J} \op
( F_j )$  consisting of the family $(P_j)$ satisfying
the following condition: 
there exists $\ell$ such that for any  complex
coordinate system $(U, z^1, \ldots , z^n)$, there exists a family $(a_\al)_{|\al | \leqslant
  \ell}$ of $\Ci (U )$ such that we have over $U$
$$ P_j = \sum_\al a_\al \bigl( \nabla_1^{F_i}\bigr) ^{\al (1) } \ldots \bigl( \nabla_n^{F_i}\bigr) ^{\al (n) }  
, \qquad \forall j \in J.$$ 
Here $\nabla^{F_j}$ is the Chern connection of $F_j$, and
$\nabla_{\ell}^{F_{j}}$ is the covariant derivative with respect to
$\partial_{ z^\ell}$. To check that $\op_{sc} (F)$ is a subalgebra,
it suffices to use that  $[\nabla_{k}^{F_{j}}, \nabla_{\ell}^{F_j} ]=
0$. 

Let $\mu$ be a measure of $M$. We define a scalar product on $\Ga ( M,
F_j)$ by integrating the pointwise scalar product of sections against
$\mu$. Denote by $\Pi_j$ the orthogonal projector of $\Ga (M , F_j)$
onto its subspace of holomorphic sections $H^0 ( M ,
F_j)$. 

\begin{prop}  \label{prop:algebra-toeplitz}
For any $(P_j) \in \op_{sc} (F)$, there exists a function $f \in \Ci
(M)$ such that for any $j \in J$
$$\Pi_j P_j \Pi_j = \Pi_j M_f \Pi_j $$
where $M_f$ is multiplication operator of $\Ga (M, F_j)$ with
multiplicator $f$.  
\end{prop}

The proof is based on a trick due to Tuynman \cite{Tuy}, a similar result was also used in \cite{An}. 

\begin{proof} 
Let $(P_j) \in \op (F) $. Observe that all the $P_j$ have the same
order $\ell$ and the same symbol $\si$. This symbol is a section of
$S^\ell ( T^{1,0} M)$. Using a partition of unity, we can write $\si$
under the form
$$ \si = \sum_{i =1, \ldots , r } X^i_1 \otimes \ldots \otimes
X^i_\ell $$
where the $X^i _k$ are smooth sections of $ T^{1,0} M$. Since 
$$ P_j - \sum_{i = 1, \ldots ,r} \nabla^{F_j}_{X^i_1} \ldots \nabla^{F_j}_{X^i_\ell}  $$
has order $\ell -1$, it suffices to prove the proposition for each
operator $$\nabla^{F_j}_{X^i_1} \ldots \nabla^{F_j}_{X^i_\ell} .$$
Let $s_1, s_2$ be smooth sections of $F_j$. Assume that $s_2$ is
holomorphic. Then for any smooth section $X$ of $T^{1,0} M$, we have
$$ X. [ (s_1, s_2) \mu ] = ( \nabla^{F_i}_X s_1, s_2) \mu +  (s_1,
s_2) ( \operatorname{div} X) \mu$$   
So if we denote by $(\cdot , \cdot )_{F_j}$ the scalar product of
sections we obtain
$$  ( \nabla^{F_j}_X s_1 , s_2 )_{F_j} = ( f s_1 , s_2 ) _{F_j} 
$$
with $f = - \operatorname{div} X$. So if $X_1, \ldots , X_\ell$ are
smooth sections of $T^{1,0} M$, then 
\begin{xalignat*}{2} 
 ( \nabla^{F_j}_{X_1} \ldots \nabla^{F_j}_{X_\ell} s_1 , s_2) _{F_j}
= &  ( f_1  \nabla^{F_j}_{X_2}\ldots \nabla^{F_j}_{X_\ell} s_1 , s_2)
_{F_j} \\ 
= &  (   \nabla^{F_j}_{f_1 X_2} \ldots \nabla^{F_j}_{X_\ell} s_1 , s_2)
_{F_j} \\ 
 & \ldots \\
= & ( f s_1 , s_2)_{F_j}
\end{xalignat*} 
where $f$ is a function which depends only on $\mu$ and the vector
fields $X_i$. This proves the result. 
\end{proof}

\section{Asymptotic flatness}\label{sec:asymptotic-flatness}

In this part we consider the same data as in section
\ref{sec:settting}. 
We will prove that the curvature of the connection $\nabla^{\Hilb_k}$
defined in theorem \ref{theo:Hitchin} vanishes in the semi-classical
limit $k \rightarrow \infty$.  For any vector field $X$ of $U$, denote by $P_k(X)$ the operator
$$ P_k(X) = \tfrac{1}{4} ( \Delta^{G(X)} - H(X)) $$
The curvature of (\ref{eq:def_connection_Hitchin}) in the directions
$X, Y \in \Ga (U,TU)$ is  
\begin{xalignat*}{2} 
   R_k(X,Y) = &  \bigl[ \nabla_{X}^{L^k \otimes \delta} +
k^{-1} P_k(X),\nabla_{Y}^{L^k \otimes \delta} + k^{-1} P_k(Y) \bigr] \\ &  -
\nabla_{[X,Y]}^{L^k \otimes \delta} - k^{-1} P_k([X,Y]).
\end{xalignat*}
Recall that the algebra of differential operators acting on the
sections of a holomorphic fiber bundle is the direct sum of the
algebra of holomorphic differential operators and the left ideal
generated by the anti-holomorphic vector fields. 
If $Q_u$ is a differential operator acting on $\Ga ( M_u, L_u
^k \otimes \delta_u)$, we denote by $Q_u^{\hol}$ its holomorphic
part. We use the same notations for families $(Q_u)_{ u \in U}$.  
We denote by $\op _{sc}$ the space of families 
$$( Q_{u,k}: \Ga ( M_u, L_u
^k \otimes \delta_u) \rightarrow \Ga ( M_u, L_u
^k \otimes \delta_u) )_{u \in U, k \geqslant  k_0}$$ consisting of differential operators such that for any $k$, $Q_{u,k}$ depends smoothly on $u$ and for any $u$, $(Q_{u,k}) _{k}$ belongs to the
algebra $\op_{sc}
(L_u^k \otimes \delta_u, k \geqslant k_0)$ introduced in the previous
section. 

\begin{theo} \label{theo:curvature_Hitchin}
For any vector fields $X,Y$ of $U$, one has
$$  R_k(X,Y) ^{\hol} = k^{-1} P_{1,k}(X,Y) + k^{-2} P_{2,k} (X,Y),
\qquad 
\forall k \geqslant k_0 $$
where the families $(P_{1,k}(X,Y))_k $ and $(P_{2,k}(X,Y))_k $ belong
to $\op_{sc}$.
\end{theo}

Theorem \ref{theo:curvature} about the asymptotic flatness
follows. Indeed by proposition \ref{prop:algebra-toeplitz}, for any
family $(Q_{u,k})$ of $\op_{sc}$, there exists a continuous function
$C : U \rightarrow \R$ such that for any $k$ and $u$, the uniform norm of
$$\Pi_{u,k} Q_{u,k}: H^0 ( M_u , L_u ^k \otimes \delta_u ) \rightarrow  H^0 ( M_u , L_u ^k \otimes \delta_u ) $$
 is bounded by $C(u)$. 

The remainder of this section is devoted to the proof of theorem \ref{theo:curvature_Hitchin}. 
Since $L$ is the pull-back of a bundle over $M$, its curvature in the
directions tangent to $U$ vanishes. For the half-form bundle, the
curvature $R^{\delta}$ depends on the derivative of the complex
structure. Recall that we denote by $\mu (X)$ the variation of the
complex structure, cf. section \ref{sec:vari-compl-struct}.

\begin{prop}  \label{prop:courbure_demi_formes_2}
For any vector field $X,Y$ of $U$, 
$$ R^{\delta}  (X,Y) = \frac{1}{8} \operatorname{tr} ( \mu (X)
\overline{\mu}(Y) - \mu (Y) \overline{\mu} (X) )  
.$$ 
\end{prop}

The proof is easy, cf. as instance the proof of theorem 7.2 in \cite{oim_mc}.

\begin{prop} \label{prop:res_int}
For any vector fields $X,Y$ of $U$ we have
$$ \bigl[ \nabla_X^{L^k \otimes \delta} , P_k (Y) \bigr]^{\hol} =
\frac{k}{8} \operatorname{tr} ( \mu (Y) \overline{\mu} (X) ) + P_k
(X,Y)$$
where the family $(P_k(X,Y))$ belongs to $\op_{sc}$.  
\end{prop}

\begin{proof} 
Introduce a local frame $\partial_1 , \ldots, \partial_n$ of the
relative holomorphic tangent bundle of $ M \times U$. Denote by
$\nabla^k$ the covariant derivative of $L^k \otimes \delta$ and by
$\nabla_i^k$ the covariant derivative in the direction of
$\partial_i$. In the sequel, repeated indices $i$ and $j$ are summed over. 
We have 
\begin{gather}\label{eq:delta}
 \Delta^{G(Y)}  = f_j \nabla_j^k + G_{i,j} \nabla_i ^k \nabla_j^k .
\end{gather}
where the coefficients  $f_i$ and $G_{ij}$ do not depend on $k$. Then
\begin{gather} \label{eq:br1} 
 \bigl[ \nabla_X^k , f_j \nabla^k _j \bigr] = (X. f_j) \nabla_j ^k
+ f_j  \nabla ^k _ {[X, \partial_j]} + f_j R^k (X, \partial_j) 
\end{gather}
where $R^k$ is the curvature of $\nabla^k$. 
The first term of the right hand side clearly belongs to
$\op_{sc}$, the third term also because $$R^k (X, \partial_j) =
R^{\delta} (X, \partial_j)$$ is independent of $k$.  For the second term, 
observe that the
holomorphic part of $\nabla^k_{[X, \partial_j]}$ is a linear combination
of the $\nabla^k_j$ with smooth coefficients which do not depend on
$k$. So the holomorphic part of (\ref{eq:br1}) belongs to $\op_{sc}$.
Let us compute the bracket of $\nabla_X^k$ with the second term of
(\ref{eq:delta}). 
$$ \bigl[ \nabla ^k_X , G_{ij}   \nabla_i ^k \nabla_j ^k \bigr] =
(X. G_{ij} )  \nabla_i ^k \nabla_j ^k + G_{ij} \bigl[ \nabla_X ^k ,
\nabla_i^k \bigr] \nabla_j ^k +  G_{ij}  \nabla_i ^k \bigr[
\nabla_X^k , \nabla_j ^k \bigr] $$
The first term of the right hand side belongs to $\op_{sc}$. The same
holds for the
holomorphic part of the third term because
$$ G_{ij}  \nabla_i ^k \bigr[
\nabla_X^k , \nabla_j ^k \bigr] = G_{ij} \nabla_i^k \nabla^k _{[X,
  \partial_j] } + G_{ij} \nabla_i ^k R^k (X, \partial_j).$$
and we can argue as we did for (\ref{eq:br1}).  
The second term is equal to 
\begin{xalignat*}{2} 
 G_{ij} \bigl[ \nabla_X ^k ,
\nabla_i^k \bigr] \nabla_j ^k =  & G_{ij} \nabla^k _{[X,\partial_i] }
\nabla_j ^k  + G_{ij} R^k ( X, \partial_i) \nabla_j ^k \\
= & G_{ij}\nabla_j ^k \nabla^k _{[X,\partial_i] } + G_{ij} \nabla  ^k
_{ [ [ X, \partial_i ] ,\partial_j ]} + G_{ij} R^k ( [ X, \partial_i ] ,
\partial_j )\\ &  + G_{ij} R^k ( X, \partial_i ) \nabla_j ^k 
\end{xalignat*}
All the terms of this last sum have a holomorphic part in $\op_{sc}$
except the third one which is equal to 
$$ G_{ij} R^k ( [ X, \partial_i ] ,
\partial_j ) = G_{ij} R^\delta ( [ X, \partial_i ] ,
\partial_j ) +  \frac{k}{i}G_{ij} \om ( [ X, \partial_i ] ,
\partial_j )
$$
Since $ \mu (Y) = G_{ij} \om ( \partial_i, \cdot ) \partial_j$, we
have
$$ \operatorname{tr} ( \mu (Y) \overline{\mu }(X) ) = G_{ij} \om (
\partial_i, \overline{ \mu} (X) ( \partial_j) ) $$ 
Using that $ [X, \partial_i ] = -  \frac{i}{2} \overline{\mu} (X) (
\partial_i) $ modulo $E$, it follows that
$$ \frac{k}{i} G_{ij}   \om ( [ X, \partial_i ] ,
\partial_j ) = \frac{k}{2} \operatorname{tr} ( \mu (Y) \overline{\mu }(X) )
$$ 
Collecting the various terms, we obtain the result. 
\end{proof}

Let us conclude the proof of theorem \ref{theo:curvature_Hitchin}. We
have
\begin{xalignat*}{2}
 R_k (X, Y) =  & R^\delta (X,Y) + \frac{1}{k} \bigl[ \nabla_X^{ L^k
  \otimes \delta}, P_k (Y) \bigr] -  \frac{1}{k} \bigl[ \nabla_Y^{ L^k
  \otimes \delta} , P_k (X) \bigr]  \\ & + \frac{1}{k^2} \bigl[ P_{k} (X) ,
P_k (Y) \bigr]  - \frac{1}{k} P_k ([X,Y])
\end{xalignat*}
By propositions \ref{prop:courbure_demi_formes_2} and \ref{prop:res_int},
the holomorphic part of the sum of the first three terms is in $k^{-1} \op_{sc}$. The last two
terms  belong respectively to $k^{-2} \op_{sc} $ and
$k^{-1} \op_{sc}$. 

\section{Semi-classical connection} \label{sec:semi-class-conn}
 
Let $(M, \om)$ be a compact symplectic manifold with a prequantum
bundle $L_M \rightarrow M$. Consider a manifold
$U$ and a smooth family $(j_u, \delta_u , \varphi_u)_{u \in U}$ consisting of positive complex structures with half-form bundles. 

We adopt the same notations and conventions as in sections \ref{sec:vari-compl-struct}
and \ref{sec:settting}. Namely, $M_u = \{ u \} \times M$
is endowed with the complex structure $j_u$, $L_u \rightarrow M_u$ is
the prequantum bundle with the holomorphic structure induced by $j_u$.  
We denote by $L$, $\delta$ and $E$ the bundles over $U \times M$
whose restrictions to each  $M_u$ are $L_u$, $\delta_u$ and 
$T^{1,0} M_u$. Let $\Hilb_k$ be the vector bundle over $U$ whose fibers are the Hilbert spaces
$$ \Hilb_{k, u } = H ^{0} (M_u , L_u ^k \otimes \delta_u) $$
and denote by $\Pi_{k,u}$ the orthogonal projector
 from $\Ga ( M_u ,  L_u^k \otimes \delta_u )$ onto $\Hilb_{k, u}$.

We now define a connection of the bundle $\Hilb_k$. Consider the
same connections on $\delta $ and $L$ as in section
\ref{sec:settting}. We set for any vector field $X$ of $U$ and section $s$ of $\Hilb_k$
$$ (\nabla_X^{\toep , k} s  ) (u) : =  \Pi_{k,u}\bigl(  (\nabla
_X^{L^k \otimes \delta}s)(u) \bigr)   $$
It is easily proved that this is indeed a connection. More generally
we shall consider the connections 
$$\nabla ^{\toep, k} +A_k, \qquad A_k  \in \Om^1 (U, \End (\Hilb_k) )$$
where the family $(A_k ,k=1,2, \ldots)$ is a  Toeplitz operator. This has the following meaning. Let $p$ be the projection from $U \times M$ onto $U$. Then there exists a sequence $f( \cdot, k)$ of $ \Ga (U\times M, p^* (T^*U \otimes \C) )$ admitting an asymptotic expansion of the form $f_0 + k^{-1} f_1 + \ldots$ for the $\Ci$-topology on the compact subsets, such that 
$$ A_k (X) (u) = \Pi_{k,u} M _{f(X) ( \cdot, u ,k )}  : \Hilb_{k, u} \rightarrow \Hilb_{k, u}$$
for any vector field $X$ of $M$. Here $M_g$ denote the multiplication operator by
$g$. We call $f_0$ the principal symbol of $(A_k)$. 
This includes the
connection defined in section \ref{sec:existence}. Indeed
proposition \ref{prop:algebra-toeplitz} implies the 
\begin{prop} There exists $f_1 \in \Ga (U\times M, \pi^* (T^*U)
  \otimes \C )$ such that for any $k \geqslant k_0$, we have  
$$ \nabla^{\Hilb_k}_X = \nabla^{\toep,k }_X+ k^{-1} \Pi_k M_{f_1 (X)}
$$ for any vector field $X$ of $U$.
\end{prop}
The results in \cite{An} also rely on the comparison between $\nabla^{\Hilb_k}_X$ and $\nabla^{\toep,k }_X$ (without the metaplectic correction).

The following theorem says that the curvature of these connections is a 
Toeplitz operator in a semi-classical sense. 

\begin{theo}  \label{theo:curvature_semi_classique}
There exists a sequence $g ( \cdot, k) \in \Ga (U\times M, p^* (
\wedge^2 T^*U \otimes \C ))$ admitting an asymptotic
expansion of the form $g_0 + k^{-1} g_1 + \ldots$ for the
$\Ci$-topology on compact subsets, such that the curvature of
$\nabla^{\toep, k} + A_k$ satisfies
\begin{xalignat*}{2}
 R^{A_k} (X,Y)_u =  \Pi_{k,u} M_{ g(X,Y)(\cdot ,u,k) } + O( k^{-\infty}) 
\end{xalignat*}
where the $O(k^{-\infty})$ is uniform on compact set of
$U$. Furthermore, $g_0$ is given by 
$$ g_0(X,Y) = X.  f_0  (Y) - Y . f_0 (X) - f_0 ([X,Y]) $$
with $f_0$ the principal symbol of $(A_k)$.
\end{theo}
This theorem is a generalization of theorem 7.1 of \cite{oim_mc},
where we did not consider the terms $(A_k)$. The proof is an immediate
generalization of the one in $\cite{oim_mc}$. This provides another
proof of theorem \ref{theo:curvature}. 

To describe the parallel transport along a path $\ga$ in the bundle
$\Hilb_k$ we introduce the notion of half-form isomorphism.  For any $u , u' \in U$ and $x \in M$ denote by $\pi_{u',u,x}$ the projection from $E_{u',x}$ to
$E_{u,x}$ with kernel $\overline E_{u',x} $.
We say that a linear isomorphism $\Psi$ of $\Hom (\delta_{u,x},
\delta_{u',x})$ is a {\em half-form isomorphism} if its square is the pull-back by $\pi_{u ' , u ,x }$, more precisely 
$$  \varphi_{u', x} \circ \Psi ^2 = \pi_{u ' , u ,x }^* \circ  \varphi_{u , x }. $$  
Such an isomorphism is unique up to a plus or minus sign.  If $\ga$ is a path of $U$, then for any $x$, there exists a unique continuous path of half-form morphism $\delta_{\ga(0) ,x } \rightarrow \delta_{\ga (t) ,x }$ starting from the identity. We denote by $\Psi (\ga)$ the morphism $\delta_{\ga(0)} \rightarrow \delta_{\ga(1)}$ obtained at $t =1$.

\begin{theo} \label{theo:paral_transp}
Let $u, u' \in U$ and $\ga$ be a path from $u$ to $u'$. For any $k
\geqslant k_0$, let $T_k :
\Hilb_{k,u} \rightarrow \Hilb_{k, u'}$ be the parallel transport along $\ga$ in the
bundle $\Hilb_k$ for the connection $\nabla^{\toep,k}+A_k$. Then the
Schwartz kernels of the operators $T_k$ have the following form 
$$ T_k (x,y) = \Bigl( \frac{k}{2 \pi } \Bigr)^{n} F^k(x,y) \otimes f (x,y , k) + O( k^{-\infty})$$ 
where $n$ is half the dimension of $M$ and 
\begin{itemize}
\item $F$ is a section of $L \boxtimes \overline {L}$ such that $| F(x,y)
| <1 $ if $x \neq y$, $ F (x,x) = v \otimes \bar{v}$ for all $x$ and $v \in L_x$ of norm 1, and $ \bar{\partial}_{j_{u'} \times -j_{u}} F \equiv 0 $
modulo a section vanishing to any order along the diagonal.
\item
  $f(.,k)$ is a sequence of sections of  $ \delta_{u'}  \boxtimes  \bar{\delta}_u
\rightarrow M^2$ which admits an asymptotic expansion in the $\Ci$
  topology of the form 
$$ h(.,k) = h_0 + k^{-1} h_1 + k^{-2} h_2 + ...$$
whose coefficients satisfy $\bar{\partial} _{j_{u'} \times -j_{u}} f_i \equiv 0  $
modulo a section vanishing to any order along the
diagonal. 
\item If the principal symbol of $(A_k)$ vanishes, then $h_0 (x,x) =
  \Psi ( \ga) .v \otimes \bar{v}$ for any $x \in M$ and $v \in
  \delta_{u,x}$ with norm 1.
\end{itemize}
\end{theo}

For $A_k = 0$, this is theorem 7.1 of \cite{oim_mc}. In the next section we provide  a proof of theorem
\ref{theo:paral_transp} more direct than the one in \cite{oim_mc} and
which shows clearly the role of the half-form bundle.

\section{Parallel transport}  \label{sec:parallel-transport}

This section is devoted to the proof of theorem
\ref{theo:paral_transp}. As previously, assume that $(M, \om)$ is a
symplectic manifold with a prequantum bundle $L_M$. 
Choose an open interval $I \subset \R$ for the parameter space. Introduce a smooth family $(j_t)_{t
  \in I}$ of
positive complex structures of $(M, \om)$. We use the same notation as before. For
instance, $M_t$ is the K{\"a}hler manifold $(M, \om ,
j_t)$ and $L_t$ is the prequantum bundle with its holomorphic
structure induced by $j_t$. Instead of half-form bundles we
begin with a smooth family $(H_t \rightarrow M_t)_{t \in I}$ of holomorphic
Hermitian line bundles. Denote by $\Hilb_{k,t}$ the space $H^0(M_t,
L^k_t \otimes H_t)$ and by $\Pi_{k,t}$ the orthogonal projector
from $\Ga (M_t , L^k_t \otimes H_t)$ onto $\Hilb_{k,t}$.

Let $\Lambda$ be a closed Lagrangian submanifold of $(M, \om)$ and $s$
be a flat unitary section of the restriction of the prequantum bundle to
$\Lambda$. Introduce a section $F$ of $L \rightarrow I \times M$ such that
$$ F( t,x ) = s (x) , \qquad \forall x \in \Lambda$$  
the pointwise norm of $F$ is $<1$ outside
$I \times \Lambda $ and $\overline \partial F$ vanishes to any order
along $ I \times \Lambda$. It is not obvious but nevertheless true
that such a section exists (proposition 2.1 of \cite{oim_qm}). It is unique up to a section
vanishing to any order along $\Lambda$. 

We say that a sequence $(f(\cdot, k))_k$ of $\Ga ( I \times M, H)$
is a symbol if it admits an asymptotic expansion
for the $\Ci $ topology on compact subsets of the form $f_0 + k^{-1} f_1
+ \ldots$ with coefficients in $ \Ga ( I \times M, H)$. We
call $f_0$ the leading coefficient of $f( \cdot , k)$ even if it
vanishes. Recall the following basic result (lemma 2.5 of \cite{oim_qm}). 

\begin{theo} \label{theo:proj_sec_lag} 
For any symbol $f( \cdot, k)$ of $\Ga ( I \times M, H)$ we have
$$\Pi_k (F^k f( \cdot , k ))  = F^k g ( \cdot , k ) + O(k^{-\infty})$$
where $g( \cdot, k )$ is a symbol of $\Ga ( I \times M,
H)$. Furthermore the restrictions to $\Lambda$ of the leading
coefficients of $f(\cdot, k)$ and $g ( \cdot, k )$ are equal. 
\end{theo}

Another important fact is that a family $( \Psi_k \in \Ga(U,
\Hilb_k))$ of the form $$ \Psi _k = F^k f( \cdot,
k) + O(k^{-\infty})$$ is determined by the restriction of the symbol
$f( \cdot , k)$ to $I \times \Lambda$. Indeed one proves that $\Psi_k$ is  $O(k^{-\infty})$ if
and only if the restriction of $f( \cdot, k)$ to $I \times \Lambda$ is 
$O(k^{-\infty})$ (lemma 1 of \cite{oim_bt})

The next theorem involves
the function $c$ of $I \times \Lambda$  defined by 
\begin{gather} \label{eq:def_c}
 c(t,x) = \frac{1}{4} \sum_{i = 1}^{n} \om ( \mu_{t,x} (
\overline{\partial}_i) , \overline \partial _i ) , \qquad  (t,x) \in I
\times \Lambda
\end{gather}
Here $(\partial_1, \ldots , \partial_n)$ is a basis of $T^{1,0}_x M_t $
such that $\partial_i + \overline{\partial}_i$ is tangent to $\Lambda$
and $\frac{1}{i} \om ( \partial_i ,\bar{\partial}_j) =
\delta_{ij}$ for any indices $i$ and $j$. As in section
\ref{sec:vari-compl-struct}, $\mu$ is the map such that 
$$\frac{d}{dt} j_t
= \mu_{t,x} + \bar{\mu}_{t,x}, \qquad \mu_{t,x} : T^{0,1}_x M_t \rightarrow T^{1,0}_x M_t$$
Introduce a connection on the bundle $H \rightarrow I \times M$. 

\begin{theo} \label{theo:proj_sans_demi_form}
For any symbol $f( \cdot, k) $ of $\Ga
  (I \times M ,
  H)$, we have that 
$$ \Pi_k ( \nabla^{L^k \otimes H} _{\partial_t} ( F^k  f(\cdot,
k))) = F^k  g(\cdot ,k) + O( k^{-\infty})$$  
where $g( \cdot, k)$ is a symbol of $\Ga ( I \times M ,
  H)$. Furthermore, the leading coefficients $f_0$ and $g_0$ satisfy 
$$  g_0 (t,x) = ( \nabla^H_{\partial_t} f_0) (t,x) - c(t,x)f_0 (t,x)  , \qquad
\forall (t,x) \in I \times \Lambda .$$
\end{theo}

\begin{proof} 
By proposition 9.4 in \cite{oim_mc}, we have $ \nabla^L_{\partial_t} F
= h F $ where $h$ is a function on $I \times M$ that vanishes along $I
\times \Lambda$. The derivatives of $h$ also vanish along $I \times
\Lambda$ and the second derivatives satisfy
$$ \bar Z_1 .\bar  Z_2 . h (t,x) = - \frac{1}{2} \om ( \bar  Z_1,
\mu ( \bar  Z_2)) , \qquad (t,x) \in I \times \Lambda $$
for any vectors $Z_1, Z_2 \in T^{1,0}_x M_t $. Then we have
$$ \nabla_{\partial_t}^{L^k \otimes H} ( F^k  f(\cdot, k )) = F^k ( k
h f(\cdot , k) + \nabla^H_{\partial_t} f(\cdot , k) ) $$
Next by theorem 4.1 of \cite{oim_hf} and theorem \ref{theo:proj_sec_lag}, we have
$$  \Pi_k \nabla_{\partial_t}^{L^k \otimes H} ( F^k f(\cdot, k ) ) = F^k
 g(\cdot ,k) + O( k^{-\infty})$$  
for a symbol $g( \cdot, k) = g_0 + O(k^{-1})$ with
$$ g_0  = -\frac{1}{2} f_0 \sum_{i=1}^n \bar \partial_i \bar
\partial_i h  + \nabla^H_{\partial_t} f_0 $$
at any $(t,x) \in I \times \Lambda$. Here the basis $(\partial_i)_i$
of $T^{1,0} _x M_t$ satisfies the same conditions as the one in the definition of $c(t,x)$.  
\end{proof}

Let $\pi = \frac{1}{2} ( \id - ij )$ be the projection of $E\oplus
\bar E$ onto $E$ with kernel $\bar E$.  
For any $(t,x) \in I \times \Lambda $,  $\pi_{t,x}$ restricts to an
isomorphism from $T_x \Lambda \otimes \C$ onto $T_x^{1,0} M_t$. So the
bundles $p^* (T \Lambda \otimes \C)$ and $\iota^*
E$ are naturally isomorphic, where $p$ and $\iota$ are respectively the projection $I \times
\Lambda \rightarrow \Lambda$ and the injection $I \times \Lambda
\rightarrow I \times M$. Consequently
\begin{gather} \label{eq:iso}
 p^* \wedge^{\top} (T ^* \Lambda \otimes \C) \simeq \iota^* \wedge^{\top} E^*
\end{gather}
Denote by $D_t$ the obvious derivation of the
sections of $p^* \wedge ^{\top} (T ^* \Lambda \otimes \C) $ with
respect to $\partial_t$. Recall that $\nabla^E _{\partial_t} $ is the
derivative of $\Ga(I \times M, E)$ obtained by composing the obvious
derivation in $E \oplus \bar E$ with respect to $\partial_t$ with
the projection $\pi$.  

\begin{prop} Identifying  the sections of $\iota^* \wedge ^{\top}E^*$ with the sections
  of $p^* \wedge ^{\top} (T ^*\Lambda \otimes \C)$ through the
  isomorphism (\ref{eq:iso}) induced by
  $\pi$, we have
$$  D_t = \nabla_{\partial_t}^{ \iota^*( \wedge ^{\top}E^*) } - 2 c $$
with $c$ the function defined in (\ref{eq:def_c}).
\end{prop} 
\begin{proof} 
Let $x$ be a point of $\Lambda$. Consider a smooth curve $t
\rightarrow U(t)$ of
$T_x \Lambda$. We have 
 $$ \nabla^E_{\partial_t} \bigl( \pi_{t,x} U(t)\bigr) = \pi_{t,x} ( \dot{\pi}_{t,x}
 U(t) + \pi_{t,x} \dot U(t)). $$
Using that $\pi = \frac{1}{2} ( \id - ij )$, we prove that
$$ \dot{\pi}_{t,x} = - \frac{i}{2} ( \mu_{t,x}  \circ \bar{\pi}_{t,x} +
\bar{\mu}_{t,x} \circ  \pi_{t,x})$$
Consequently
$$ \nabla^E_{\partial_t} \bigl( \pi_{t,x} U(t)\bigr) = - \frac{i}{2}  \mu_{t,x}  \circ \bar{\pi}_{t,x} (U(t)) + \pi_{t,x} (\dot U(t))$$
To end the proof  we have  to show that 
$$ 
\frac{i}{2} \operatorname{tr} \bigl( \mu_{t,x}  \circ \bar{\pi}_{t,x} \circ
q_{t,x} : T_x^{1,0}M_t \rightarrow T_x^{1,0} M_t \bigr)  = 2c(t,x) $$
 where $q_{t,x} : T_x^{1,0} M \rightarrow T_x
\Lambda \otimes \C$ is the inverse of the restriction
of $\pi_{t,x}$ to $T_x \Lambda \otimes \C $. Introduce a basis
$\partial_1, \ldots , \partial_n$ of $T_x^{1,0}M \otimes \C$ as in the
definition (\ref{eq:def_c}) of $c$. Then $q_{t,x} ( \partial_i ) =
\partial_i  +
\bar{\partial_i}$, so that 
$   \bar{\pi}_{t,x} ( q_{t,x}( \partial_i ))
= \bar  \partial_i  $
and then 
\begin{xalignat*}{2}
  \operatorname{tr} \bigl( \mu_{t,x}  \circ \bar{\pi}_{t,x} \circ
q_{t,x} : T_x^{1,0}M_t \rightarrow T_x^{1,0} M_t \bigr) =  & \frac{1}{i}
\sum_j \om (  \mu_{t,x} (\bar  \partial_j) , \bar \partial _j
) \\ = & 
\frac{4}{i} c( t,x) 
\end{xalignat*}
which concludes the proof. 
\end{proof}

Let us consider a family of half-form bundles $( \delta_t ,
\varphi_t)$. For any $(t,x) \in I \times \Lambda$, one has an isomorphism
$$ \tilde{\varphi}_{t,x} =   \pi_{t,x} ^* \varphi_{t,x} : \delta_{t,x}
^{\otimes 2} \rightarrow \wedge^{\top} (T\Lambda \otimes
\C)^* $$
So we have a natural derivation $D_t $ of the sections of $\iota^*
\delta$ satisfying $$ D_t  \tilde{\varphi} ( s^{\otimes 2}) = 2 \tilde \varphi(
s \otimes D_t s).$$ 
Let us apply theorem \ref{theo:proj_sans_demi_form} with the family
$(H_t)=(  \delta_t)$ and the connection $\nabla^{\delta}$ induced by
$\nabla^E$. Then by the last proposition, the equation determining the leading coefficient is
$$ \iota^* g_0 = D_t (\iota^* f_0) .$$

\begin{theo} \label{theo:parallel-transport}
Let $f(\cdot, k) $ be a symbol of $\Ci (
  I \times M)$. Let $(\Psi_k \in \Ga ( I , \Hilb_k))_k$  be a family satisfying 
$$ \Pi_k ( \nabla^{L^k \otimes \delta}_{\partial_t} \Psi_k  + M_{f( \cdot, k)}
\Psi_k) = 0$$
Assume that there exists $t_0 \in I$ and a symbol $g(\cdot, k ) $ of $\Ga ( M,
\delta_{t_0})$ such that 
$ \Psi_{k} (t_0) = F^k ( t_0, \cdot) g( \cdot , k )  + O(k^{-\infty})
. $
Then there exists a symbol $h( \cdot, k)$ of $\Ga( I \times M ,
\delta)$ such
that 
$$  \Psi_k = F^k  h( \cdot , k )  + O(k^{-\infty})
. $$
Furthermore the  leading
coefficient $h_0$ satisfies  
$$ D_t \iota^* h_0 = \iota^* (f_0 h_0)  , \qquad  h_0 ( t_0,x) =
g_0(x), \;  \forall (t,x) \in I \times \Lambda$$ 
where $f_0$ and $g_0$ are the leading coefficients of $f( \cdot, k)$
and $g( \cdot, k)$ respectively.
\end{theo}

\begin{proof} Since the argument is very standard, we only sketch it. 
Using theorems \ref{theo:proj_sec_lag} and
\ref{theo:proj_sans_demi_form}, we construct by successive
approximations a symbol $h( \cdot , k )$ satisfying
$$ \Pi_k \bigl( ( \nabla^{L^k \otimes \delta}_{\partial_t} + M_{f( \cdot, k)} )
(F^k h( \cdot, k)  ) \bigr) = O(k^{-\infty})$$
and the initial condition 
$$h(t_0 , x , k) = g ( x, k ) + O(k^{-\infty}), \qquad \forall x \in \Lambda$$ 
Then one proves that this approximate solution is equal to $(\Psi_k)$
up to a $O(k^{-\infty})$ term. 
\end{proof}

We are now ready to prove theorem \ref{theo:paral_transp}. Assume that
$0$ belongs to $I$. Denote by $U_t : \Hilb_{k,0} \rightarrow
\Hilb_{k,t}$ the parallel transport map for the connection 
$$ \Pi_k \bigl( \nabla_{\partial_t}^{L^k \otimes \delta} + M_{f(\cdot
  , k )} \bigr) $$
Consider the symplectic manifold $M' = M \times M^-$ with the
prequantum bundle $L_{M'} = L_M \boxtimes \bar L_M$, the family of complex
structures $( j_t'=j_t \times -j_0 )_{t \in I} $ and the family of half-form
bundles $(\delta'_t = \delta_t
\otimes \bar \delta_0)_{t \in I}$. We may apply all the previous
constructions to these data. We denote  them in the same way with a prime
added. 
In particular, $\Hilb'_k \rightarrow I$ is the associated
Hilbert space bundle. Its fiber at $t$ is 
\begin{xalignat*}{2} 
\Hilb'_{k,t} = & H^0 \bigl( M_t \times \bar{M}_0 , ( L^k_t \boxtimes \bar
L_0^k ) \otimes ( \delta_t \boxtimes \bar \delta _0 ) \bigr) \\ 
\simeq & \Hilb_{k,t} \otimes \overline \Hilb_{k,0} \\
\simeq & \End \bigl(  \Hilb_{k,0}, \Hilb_{k,t} \bigr)
\end{xalignat*}
Denote by $U$ the section of $\Hilb'_k$ whose value at $t$ is
$U_t$. It satisfies the following equation
$$ \Pi_k' \Bigl( \nabla^{ (L')^k \otimes \delta'  }_{\partial_t}  + M_{f'(
  \cdot , k )} \Bigr)  U = 0 
$$
where $f'(\cdot, k)$ is the pull-back of $f( \cdot, k)$ by the
projection from $I \times M^2$ onto the product of the first two
factors. We apply theorem \ref{theo:parallel-transport} to the section
$U$. The Lagrangian submanifold $\Lambda'$ is the diagonal of $M^2$,
the section $s'$ is given by $s' (x,x) = u \otimes \bar u$, for any $u
\in L_x$ of norm 1. $U_0$ is the identity of $\Hilb_{k,0}$. The fact that
it satisfies the
assumption of theorem \ref{theo:parallel-transport} is the basic result of the semiclassical-analysis
on compact K{\"a}hler manifolds. We deduced this result in \cite{oim_bt} from the
analysis of the Szeg{\"o} kernel in \cite{BoSj}. It says that the Schwartz kernel of the
identity of $\Hilb_{k,0}$ is equal to 
$$ \Bigl( \frac{k}{2 \pi} \Bigr)^n F'_0 \; g'( \cdot, k) +
O(k^{-\infty})$$
where $F'_0$ is the restriction to $M_0 \times \bar M_0$ of the
section $F'$ associated to $(\Lambda',s')$.  The leading
coefficient of $g'( \cdot, k)$ satisfies $g'(x,x,k) = u \otimes
\bar{u}+O(k^{-1}) $ for any  $u \in \delta_{0,x}$  of norm 1. 

To conclude we only have to compute the leading coefficient of
the solution. Identify the diagonal $\Lambda'$ with $M$, so that the
map $\tilde \varphi'_{t,x}$ corresponding to $\tilde \varphi_{t,x}$ sends $( \delta_{t,x}
\otimes \bar \delta_{0,x} ) ^2$ to $\wedge ^{\top} T_x^*M \otimes
\C$. Then a linear map $\xi : \delta_{0,x} \rightarrow \delta_{t,x}$
is a half-form morphism if and only if 
$$ \tilde \varphi'_{t,x} \bigl( (\xi (u)\otimes  \bar u )^2 \bigr) =
i^{n ( n-2)}  \om^{\wedge n} _x /n! $$
for all $u \in \delta_{0,x}$ of norm 1. This is the main reason at the
origin of
the definition of the half-form morphisms, cf. lemma 6.1 in
\cite{oim_mc} for a proof. The conclusion follows from the fact that
$\om_x$ doesn't depend on $t$.

\bibliography{biblio}

\end{document}